\documentclass{svjour3}                     
\smartqed  
\usepackage{graphicx}
\usepackage{amsmath,amssymb, amstext, amsfonts}

\newcommand{\su}{\operatorname{Supp}}

\newtheorem{thm}{Theorem}
\newtheorem{cor}[thm]{Corollary}

\begin{document}

\title{On supports of expansive measures}


\institute{C.A. Morales\\
Instituto de Matem\'atica, Universidade Federal do Rio de Janeiro\\
P. O. Box 68530, 21945-970 Rio de Janeiro, Brazil.\\
\email{morales@impa.br}}

\author{C.A. Morales\thanks{Partially supported by MATHAMSUB 15 MATH05-ERGOPTIM, Ergodic Optimization of Lyapunov Exponents.}}


\date{Received: date / Accepted: date}

\maketitle

\begin{abstract}
We prove that a homeomorphism of a compact metric space has an expansive measure \cite{ms} if and only if it has
many ones with invariant support.
We also study homeomorphisms for which
the expansive measures are dense in the space of Borel probability measures.
It is proved that these homeomorphisms exhibit a dense set of Borel probability measures which are
expansive with full support.
Therefore, their sets of heteroclinic points
has no interior and the spaces supporting them
have no isolated points.
\keywords{Expansive measure\and Support of a measure\and Homeomorphism.}
\end{abstract}

\section{Introduction}

\noindent
The {\em expansive homeomorphisms} were introduced by Utz in the middle of the nineteen century \cite{u}.
Since then
an extensive theory about such systems have been growing,
see \cite{lc} and references therein.

In this paper we will study the related concept of
{\em expansive measure} \cite{mo}, \cite{ms}.
These measures are closely related to the expansive systems.
Indeed, \cite{ac} proved recently that a homeomorphism is {\em measure-expansive} (i.e. every nonatomic Borel probability measure is expansive) if and only if it is {\em countably-expansivity}
(i.e. the dynamical balls of a given radio are all countable).

This study has two motivations.
The first one is Theorem 1.8 in \cite{ms} claiming that
every homeomorphism with expansive measures of a compact metric space
has an expansive invariant one.
Unfortunately, the proof in \cite{ms} is incorrect (\footnote{pointed to us by Professor S. Kryzhevich.})
so it is unknown if
every homeomorphism with expansive measures of a compact metric space has an expansive invariant one.

The second motivation is the well-known problem of finding conditions
for the ergodic measures of a given homeomorphism to be dense in the
space of invariant measures.
As a sample we can mention the classical work
by Sigmund \cite{s} proving it
for homeomorphisms satisfying Bowen's specification property.
In our case we would like to consider the analogous
problem of finding conditions for the
expansive measures to be dense in the space of Borel probability measures of
the underlying space.
Corollary 8.1 in \cite{pa} implies that this is the case for
the countably-expansive homeomorphisms on compact metric spaces without isolated points. It is then natural ask if these are the sole
homeomorphisms with this property.

Here we will obtain some partial positive answer to both questions.
Indeed, we first prove that a homeomorphism of a compact metric space has an expansive measure if and only if
it has one with invariant support.
Afterwards, we study those homeomorphisms of compact metric spaces
for which the expansive measures are dense in the set of Borel probability ones.
It is proved that all such homeomorphisms exhibit a dense set of Borel probabiltity measures which are expansive with full support.
Therefore, the set of heteroclinic points of these homeomorphisms
has no interior and the spaces supporting them
have no isolated points.
Let us state our results in a precise way.

Hereafter $X$ will denote a compact metric space.
The Borel $\sigma$-algebra of $X$ is the $\sigma$-algebra $\mathcal{B}(X)$ generated by the open subsets of $X$.
A {\em Borel probability measure} is a $\sigma$-additive measure $\mu$ defined in $\mathcal{B}(X)$ such that $\mu(X)=1$.
We denote by $\mathcal{M}(X)$ the set of all Borel probability measures of $X$.
This set is a compact metrizable convex space and its topology
is the {\em weak* topology}
defined by the convergence $\mu_n\to\mu$ if and only if
$\int\phi d\mu_n\to\int\phi d\mu$ for every continuous map $\phi: X\to \mathbb{R}$.
The {\em support} of $\mu\in\mathcal{M}(X)$ is the set $\su(\mu)$ of points $x\in X$ such that
$\mu(U)>0$ for any neighborhood $U$ of $x$. It follows that $\su(\mu)$ is a nonempty compact subset of $X$.
A measure $\mu\in\mathcal{M}(X)$ is {\em nonatomic} if $\mu(\{x\})=0$ for every $x\in X$.

Let $f: X\to X$ be a homeomorphism. We say that $\Lambda\subset X$ is {\em invariant} if $f(\Lambda)=\Lambda$.
Similarly $\mu\in \mathcal{M}(X)$ is {\em invariant} if $f_*(\mu)=\mu$ where
$f_*$ is defined by $f_*(\nu)(B)=\nu(f^{-1}(B))$ for all $B\in \mathcal{B}(X)$ and $\nu\in\mathcal{M}(X)$.
A necessary (but not sufficient) condition for a Borel probability measure to be invariant is that
its support be an invariant set.

We say that $f$ is {\em expansive} if there is $\delta>0$ such that
$\Gamma_\delta(x)=\{x\}$ where
$$
\Gamma_\delta(x)=\{y\in X:d(f^i(x),f^i(y))\leq r, \mbox{ for every } i\in\mathbb{Z}\},
\quad\quad\forall x\in X,\delta\geq0.
$$
Equivalently, $f$ is expansive if there is $\delta>0$ such that
$\Gamma_\delta(x)$ has only one element for every $x\in X$.
This motivates the following concept:
We say that $f$ is {\em countably-expansive} if
there is $\delta>0$ such that $\Gamma_\delta(x)$ is at most countable for every $x\in X$.
This concept generalizes the notion of
$n$-expansive homeomorphism considered in \cite{moo}.
The notion of expansivity has been generalized to Borel probability measures in the following way:
We say that a measure
$\mu\in \mathcal{M}(X)$ is {\em expansive}
if there is $\delta>0$ such that
$\mu(\Gamma_\delta(x))=0$ for every $x\in X$.

Clearly the set of expansive measures
is a face without extreme points of $\mathcal{M}(X)$. Further properties of this set will be given latter on.
We say that $\mu\in\mathcal{M}(X)$ is {\em fully supported} if $\su(\mu)=X$.

With these definitions and remarks we can state our result.

\begin{thm}
\label{thAA}
A homeomorphism of a compact metric space has an expansive measure if and only if it has a dense set of expansive measures with invariant support.
\end{thm}

\begin{thm}
\label{thA}
For every homeomorphism of a compact metric space $X$,
the expansive measures are dense in $\mathcal{M}(X)$ if and only if the fully supported expansive measures are dense in
$\mathcal{M}(X)$.
\end{thm}

A simple application of Theorem \ref{thA} we obtain the next corollary.
Let $f: X\to X$ be a homeomorphism.
We say that $x\in X$ is {\em periodic}
if there is $n\in\mathbb{N}^+$ such that $f^n(x)=x$.
The orbit $\{f^i(x):i\in\mathbb{Z}\}$ of a periodic point $x$ will be refereed to as a {\em periodic orbit} of $f$.
If $x\in X$ the {\em $\omega$-limit set} of $x$ is defined by
$$
\omega(x)=\left\{ y\in X:y=\lim_{k\to\infty} f^{n_k}(x)\mbox{ for some sequence }n_k\to\infty\right\}.
$$
The {\em $\alpha$-limit set} $\alpha(x)$ is the $\omega$-limit set of $x$ with respect to the inverse map $f^{-1}$. 
A {\em heteroclinic point} is a point whose $\alpha$ and $\omega$-limit sets are periodic orbits.

It is known that every countably-expansive homeomorphism of a compact metric space
satisfies that the set of heteroclinic points is countable \cite{ms}.
Moreover, if the space has no isolated points,
the expansive measures are dense in the space of Borel probability measures.
The corollary below proves a sort of converse of these results.

\begin{cor}
Let $f: X\to X$ be a homeomorphism of a compact metric space $X$.
If the expansive measures of $f$ are dense in $\mathcal{M}(X)$, then
$X$ has no isolated points and the set of heteroclinic points of $f$ has no interior points.
\end{cor}

By using this corollary we can exhibit examples of homeomorphisms for which the set of expansive measures
is a nonempty nondense subset of the space of Borel probability measures (just take a homeomorphism with an open set of fixed points).

\section{Proof of theorems \ref{thAA} and \ref{thA}}

\noindent
We shall use the following standard topological concepts.
A topological space $Y$ is a {\em Baire space} if the intersection
of each countable family of open and dense subsets in $Y$ is dense in $Y$.
A set $A\subset Y$ is:
{\em $G_\delta$ subset} of $Y$ if
it is the intersection of countably many open subsets of $Y$;
{\em $G_{\delta\sigma}$ subset} of $Y$ if
it is the union of countably many $G_\delta$ subsets of $Y$;
{\em Baire subset} if
$A$ is a Baire space with respect to the topology induced by $Y$;
{\em nowhere dense in $Y$} if the closure of $A$ in $Y$ has empty interior in $Y$; and
{\em meagre}
if it is the union of countably many nowhere dense subsets of $Y$.

It is well-known that every $G_\delta$ subset of a complete metric space $Y$ is a Baire subset of $Y$ \cite{w}.
Below we extend this property to the $G_{\delta\sigma}$ subsets.
Since we did not found any reference for this fact, we include its proof here for the sake of completeness.

\begin{lemma}
\label{bbaire}
Every $G_{\delta\sigma}$ subset $A$ of a complete metric space $Y$ is a Baire subset of $Y$.
\end{lemma}

\begin{proof}
Without loss of generality we
can assume that $A$ is dense in $Y$ (otherwise replace $Y$ by the closure of $A$).

A subset of $Y$ is said to have the {\em property of Baire}
if it is the symmetric difference between an open subset of $Y$ and a meagre subset of $Y$.
It turns out that the set formed by the subsets having the property of Baire coincides with the $\sigma$-algebra of $Y$
generated by the open subsets of $Y$ and the meagre subsets of $Y$
(see Theorem 4.3 in \cite{o}).
It is clear from the definition that every $G_{\delta\sigma}$ subset of $Y$
belongs to such a $\sigma$-algebra. Since $A$ is a $G_{\delta\sigma}$ subset, we conclude that
$A$ belongs to such a $\sigma$-algebra and so it has the property of Baire. Then, by Theorem 4.4 in \cite{o}, there are
a $G_\delta$ subset $B$ of $Y$ and a meagre subset $C$ of $Y$ such that $A=B\cup C$.

Since $B$ is $G_\delta$ in $Y$, there is a sequence of open subsets $G_n$ of $Y$ such that
$$
B=\bigcap_n G_n.
$$
We claim that $B$ is dense in $Y$.
Otherwise there is an open subset $O$ of $Y$ such that $O\cap B=\emptyset$
and so $O\cap A=O\cap C$. Since $A$ is dense in $Y$, we have that
$O\cap A$ is dense in $O$. Therefore, $O\cap C$ is dense in $O$ which contradicts that $C$ is meagre in $Y$.
Therefore, $B$ is dense in $Y$ and the claim follows.

It follows from the claim that each $G_n$ is dense in $Y$ too.

Now take a sequence $V_m$ of open and dense subsets of $A$.
Then, there is a sequence of open sets $W_m$ of $Y$ such that $V_m=W_m\cap A$.
Hence
$V_m=(W_m\cap B)\cup (W_m\cap C)$ for all $m$.
We have that $W_m\cap C$ is meagre in $Y$ and $V_m$ is dense in $A$
(and so in $X$ too).
It follows that $W_m\cap G_n$ is open-dense in $Y$ (for all $n,m$)
and so
$\bigcap_{n,m}(W_m\cap G_n)$ is dense in $Y$ by Baire's category theorem.
As
$$
\bigcap_m (B\cap W_m)=
B\cap \left(\bigcap_mW_m\right)=
\left( \bigcap_n G_n\right)\cap \left( \bigcap_m W_m\right)=
\bigcap_{m,n}(W_m\cap G_n),
$$
$\bigcap_m (B\cap W_m)$ is dense in $Y$ too.
As
$$
\bigcap_m (B\cap W_m)\subset
\bigcap_mV_m,
$$
$\bigcap_m V_m$ is dense in $Y$ and therefore in $A$. This ends the proof.
\qed
\end{proof}

To motivate the next lemma we
observe that the set of expansive measures of a given homeomorphism may be empty
(e.g. circle rotations) and if nonempty it may be noncompact.
For instance, the set of expansive measures of any expansive homeomorphism $f: X\to X$ coincides
with the set of nonatomic Borel probability measures which, in turns, is a $G_\delta$ subset of $\mathcal{M}(X)$ (see \cite{pa} or Theorem 2.2 in \cite{ls}).
Below we use the arguments in \cite{pa} to prove that the set of expansive measures of an arbitrary homeomorphism
is a $G_{\delta\sigma}$ subset. More precisely, we obtain the following result.

\begin{lemma}
\label{l1}
The set of expansive measures of a homeomorphism $f: X\to X$ of a compact metric space $X$
is a $G_{\delta\sigma}$ subset of $\mathcal{M}(X)$.
\end{lemma}

\begin{proof}
Let $\mathcal{M}_{ex}(X,f)$ denote the set of expansive measures of $f$.

For all $\delta,\epsilon>0$ we define
$$
C(\delta,\epsilon)=\{\mu\in \mathcal{M}(X):\mu(\Gamma_\delta(x))\geq\epsilon\mbox{ for some }x\in X\}.
$$
It follows that
\begin{equation}
\label{baire}
\mathcal{M}_{ex}(X,f)=
\bigcup_{n=1}^\infty\bigcap_{m=1}^\infty \left(\mathcal{M}(X)\setminus C(n^{-1},m^{-1}\right)).
\end{equation}

We claim that $\mathcal{M}(X)\setminus C\left(\delta,\epsilon\right)$ is open in $\mathcal{M}(X)$
for any $\delta,\epsilon>0$.

Take $\delta,\epsilon>0$ and a sequence $\mu_n\in C(\delta,\epsilon)$
such that $\mu_n\to \mu$ for some $\mu\in\mathcal{M}(X)$.
Choose a sequence $x_n\in X$ such that
$$
\epsilon\leq \mu_n(\Gamma_\delta(x_n)),
\quad\quad\forall n\in\mathbb{N}^+.
$$
As $X$ is compact, we can assume that $x_n\to x$ for some $x\in X$.
Fix a compact neighborhood $C$ of $\Gamma_\delta(x)$.
Denote by $O=int(C)$ the interior of $C$. Hence $\Gamma_\delta(x)\subset O$.
Suppose for a while that there is a subsequence $n_k\to\infty$ such that
$\Gamma_\delta(x_{n_k})\not\subset O$ for all $k\in\mathbb{N}$.
Then, we can select a sequence $z_k\in \Gamma_\delta(x_{n_k})\setminus O$.
Again by compactness we can assume that $z_k\to z$ for some $z\in X$.
Since $O$ is open, $z\notin O$.
However, $z_k\in\Gamma_\delta(x_{n_k})$ and so
$$
d(f^i(z_k),f^i(x_{n_k}))\leq\delta,
\quad\quad\forall i\in\mathbb{Z}.
$$
Fixing $i$ and letting $k\to\infty$ we obtain
$$
d(f^i(z),f^i(x))\leq\delta, \quad\quad \forall i\in\mathbb{Z}.
$$
Hence $z\in \Gamma_\delta(x)$.
As $\Gamma_\delta(x)\subset O$, we get $z\in O$ which is absurd.
Then,
$\Gamma_\delta(x_n)\subset C$ and so
$\mu_n(\Gamma_\delta(x_n))\leq\mu_n(C)$ for all $n$ large.
Since $\mu_n\to \mu$ we obtain
$$
\epsilon\leq \limsup_{n\to\infty}\mu_n(\Gamma_\delta(x_n))\leq\limsup_{n\to\infty}\mu_n(C)\leq\mu(C).
$$
This proves $\mu(C)\geq\epsilon$ for every closed neighborhood $C$ of $\Gamma_\delta(x)$.
Hence $\mu(\Gamma_\delta(x))\geq\epsilon$.
It follows that $C(\delta,\epsilon)$ is closed for any $\delta,\epsilon>0$ and the claim
follows.

The claim implies that
$\bigcap_{m=1}^\infty \left(\mathcal{M}(X)\setminus C(n^{-1},m^{-1}\right))$ is a $G_\delta$ subset of $\mathcal{M}(X)$ for all $n\in\mathbb{N}^+$.
Then, $\mathcal{M}_{ex}(X,f)$ is a $G_{\delta\sigma}$ subset of $\mathcal{M}(X)$ (by (\ref{baire})).
This completes the proof.
\qed
\end{proof}

Since the space $\mathcal{M}(X)$ of Borel probability measures equipped with the weak* topology of a compact metric space $X$
is a compact (hence complete),
lemmas \ref{bbaire} and \ref{l1} imply the following corollary.

\begin{cor}
\label{tudo}
The set of expansive measures equipped with the weak* topology of a homeomorphism of a compact metric space
is a Baire space.
\end{cor}

Recall that $\mathcal{M}_{ex}(X,f)$ denotes the set of expansive measures of $f$.

\begin{lemma}
\label{l2}
For every homeomorphism $f: X\to X$ with expansive measures of a compact metric space $X$ there is
a meagre subset $\mathcal{D}$ of $\mathcal{M}_{ex}(X,f)$
such that
$$
\su(\mu)=\bigcup_{\nu\in\mathcal{M}_{ex}(X,f)}\su(\nu),
\quad\quad\forall \mu\in \mathcal{M}_{ex}(X,f)\setminus \mathcal{D}.
$$
\end{lemma}

\begin{proof}
Let $2^X_c$ denote the set of compact subsets of $X$.
This set becomes a compact metric space if endowed with the Hausdorff distance.
Define $\Psi: \mathcal{M}_{ex}(X,f)\to 2^X_c$ by
$\Psi(\mu)=\su(\mu)$.
It is easy to see that $\Psi$ is lower-semicontinuous and so
the set $\mathcal{D}$ of discontinuity points of $\Psi$
is meagre (by Corollary 1 p. 71 in \cite{k1}).
Let us prove that this $\mathcal{D}$ satisfies the desired property.

Take $\mu\in \mathcal{M}_{ex}(X,f)\setminus \mathcal{D}$ and $\nu\in \mathcal{M}_{ex}(X,f)$.
Define $\mu_t=(1-t)\mu+t\nu$ for $t\in ]0,1[$.
Clearly $\mu_t\in\mathcal{M}_{ex}(X,f)$ and $\mu_t\to\mu$ as $t\to0$.
Since $\mu\in \mathcal{M}_{ex}(X,f)\setminus \mathcal{D}$, $\Psi$ is continuous at $\mu$
and so $\Psi(\mu_t)=\su(\mu_t)=\su(\mu)\cup \su(\nu)$ converges to $\Psi(\mu)=\su(\mu)$.
From this we obtain $\su(\nu)\subset \su(\mu)$ and the proof follows.
\qed
\end{proof}

Following the definition of measure center \cite{z} we introduce the concept below.

\begin{definition}
The {\em measure-expansive center} of a homeomorphism $f: X\to X$ is the union of the support
of all the expansive measures of $f$, i.e.,
$$
E(f)=\bigcup_{\nu\in\mathcal{M}_{ex}(X,f)}supp(\nu).
$$
\end{definition}

The measure-expansive center can be characterized as the set of points $x$ with the property that
for every neighborhood $U$ of $x$ there is an expansive measure $\mu$ such that $\mu(U)>0$.
We do not prove this property here since it is unuseful for our purposes.
A property that will be used is given below.

\begin{lemma}
\label{me}
For every homemorphism $f: X\to X$ with expansive measures of a compact metric space
there is a dense subset $\mathcal{R}$ of $\mathcal{M}_{ex}(X,f)$ such that $\su(\mu)=E(f)$ for all
$\mu\in \mathcal{R}$.
\end{lemma}

\begin{proof}
By Corollary \ref{tudo} we have that $\mathcal{M}_{ex}(X,f)$ equipped with the weak* topology is a Baire space.
Now let $\mathcal{D}$ be the meagre subset of $\mathcal{M}_{ex}(X,f)$ given by Lemma \ref{l2}.
Since $\mathcal{M}_{ex}(X,f)$ is Baire,
$\mathcal{M}_{ex}(X,f)\setminus \mathcal{D}$ is dense in $\mathcal{M}_{ex}(X,f)$.
Then, Lemma \ref{l2} implies the result by taking $\mathcal{R}=\mathcal{M}_{ex}(X,f)\setminus \mathcal{D}$.
\qed
\end{proof}

We can also prove that the measure-expansive center of a homeomorphism $f$
coincides with the intersection of all compact invariant subsets $E$ of $f$
for which $\mu(E)=1$ for all expansive measure $\mu$ of $f$. From this characterization we obtain immediately that
the measure-expansive center is a compact invariant set of $f$.
We will obtain the latter property directly from the following corollary.

\begin{cor}
\label{c1}
The measure-expansive center is a
(possibly empty) compact invariant set.
\end{cor}

\begin{proof}
Let $f: X\to X$ be a homeomorphism of a compact metric space $X$.
If $\mathcal{M}_{ex}(X,f)=\emptyset$ there is nothing to prove.
Otherwise, by Lemma \ref{me}, there is a dense subset of expansive measures $\mu$ satisfying $\su(\mu)=E(f)$.
Since the support of any measure is compact, we obtain that $E(f)$ is compact too.
To obtain that $E(f)$ is invariant, we simply
observe that $f(\su(\nu))=\su(f_*(\nu))$ for all $\nu\in\mathcal{M}(X)$ and that
$f_*(\nu)\in \mathcal{M}_{ex}(X,f)$ for all $\nu\in\mathcal{M}_{ex}(X,f)$ (c.f. \cite{ms}).
This completes the proof.
\qed
\end{proof}

The final ingredient is the following simple fact about dense subsets of $\mathcal{M}(X)$.

\begin{lemma}
\label{l3}
If $X$ is a compact metric space, then $\bigcup_{\nu\in\mathcal{D}}supp(\nu)$ is dense in $X$
for any dense subset $\mathcal{D}$ of $\mathcal{M}(X)$.
\end{lemma}

\begin{proof}
Otherwise, there would exist $x\in X$ and an open neighborhood $O$ of $x$ such that
$O\cap \su(\mu)=\emptyset$ for every $\nu\in \mathcal{D}$.
Since $\mathcal{D}$ is dense in $\mathcal{M}(X)$, there is a sequence $\mu_n\in\mathcal{D}$ such that
$\mu_n\to \delta_x$ where $\delta_x$ is the Dirac measure supported on $x$.
By the choice of $x$ we have $\mu_n(O)=0$ for all $n\in\mathbb{N}$.
Since $\mu_n\to \delta_x$, we have
$\delta_x(O)\leq\liminf_{n\to\infty}\mu_n(O)=0$
which is absurd.
\qed
\end{proof}

Now we can prove our results.

\begin{proof}[of Theorem \ref{thAA}]
Let $f: X\to X$ be a homeomorphism with expansive measures of a compact metric space $X$.
By Lemma \ref{me} there is a dense subset of expansive measures
whose supports are all equal to $E(f)$.
Since $E(f)$ is invariant by Corollary \ref{c1},
we are done.
\qed
\end{proof}

\begin{proof}[of Theorem \ref{thA}]
Let $f: X\to X$ be a homeomorphism of a compact metric space $X$.
Suppose that $\mathcal{M}_{ex}(X,f)$ is dense in $\mathcal{M}(X)$.
It follows from Lemma \ref{l3} that
$E(f)=X$. Then, Lemma \ref{me} provides a dense subset $\mathcal{R}$ of $\mathcal{M}_{ex}(X,f)$ such that $\su(\mu)=X$
for all $\mu\in\mathcal{R}$.
Since $\mathcal{M}_{ex}(X,f)$ is dense in $\mathcal{M}(X)$, we have that $\mathcal{R}$ is dense in $\mathcal{M}(X)$
and we are done.
\qed
\end{proof}

\begin{acknowledgements}\label{ackref}
The author would like to thank professors A. Arbieto and B. Santiago for helpful conversations.
\end{acknowledgements}

\end{document}